\documentclass[a4paper]{article}
\usepackage[english]{babel}
\usepackage[utf8]{inputenc}
\usepackage{amsmath,amsthm}
\usepackage{amsfonts,amssymb}

\usepackage[colorlinks=true,linktocpage=true,linkcolor=blue,citecolor=red]{hyperref}
\usepackage[capitalise]{cleveref} 

\usepackage{stmaryrd}
\usepackage{enumitem}
\setenumerate[1]{label=(\arabic*)} 

\usepackage{mathtools}

\usepackage{tikz,tikz-cd} 
\usetikzlibrary{arrows}

\newcommand{\Acal}{\mathcal{A}}

\newcommand{\Dcal}{\mathcal{D}}

\newcommand{\Ccal}{\mathcal{C}}

\makeatletter
\def\namedlabel#1#2{\begingroup
    #2%
    \def\@currentlabel{#2}%
    \phantomsection\label{#1}\endgroup
}
\makeatother

\swapnumbers

\theoremstyle{plain}
\newtheorem{theorem}{Theorem}[section]
\newtheorem{lemma}[theorem]{Lemma}
\newtheorem{prop}[theorem]{Proposition}
\newtheorem*{prop*}{Proposition}
\newtheorem{cor}[theorem]{Corollary}

\theoremstyle{definition}

\newtheorem{remark}[theorem]{Remark}

\newtheorem{construction}[theorem]{Construction}

\newcommand{\id}{\text{Id}}
\newcommand{\set}{\textbf{Sets}}

\newcommand{\cat}{\textbf{Cat}}

\newcommand{\Ord}{\textbf{Ord}}

\newcommand{\op}{\text{op}}

\DeclareMathOperator*{\colim}{Colim}
\DeclareMathOperator*{\Lim}{Lim}

\DeclareMathOperator{\Hom}{Hom}

\DeclareMathOperator{\Ind}{Ind}
\DeclareMathOperator{\Nat}{Nat}
\DeclareMathOperator{\Fun}{Fun}

\begin{document}

\pagestyle{plain}
\title{When does $\operatorname{Ind}_\kappa(C^I) \simeq \operatorname{Ind}_\kappa(C)^I$?}

\date{}

\author{Simon Henry}


\maketitle

\begin{abstract} We investigate under which condition the $\kappa$-ind completion of a functor category $C^I$ is equivalent to the category of functors from $I$ to the $\kappa$-ind completion of $C$. A published theorem implies this is true for any Cauchy complete category $C$ and $\kappa$-small category $I$, but we show this is not the case in general. We prove two results that seem to cover all applications of this incorrect theorem we could find in the literature: The result holds if $C$ has $\kappa$-small colimits and $I$ is $\kappa$-small, or if $C$ is an arbitrary category and $I$ is well-founded and $\kappa$-small. In both cases, we show that the conditions are optimal in the sense that the result holds for all $C$ if and only if $I$ satisfies the given assumption.
\end{abstract}

\vspace{-0.2cm}

\renewcommand{\thefootnote}{\fnsymbol{footnote}} 
\footnotetext{\emph{2020 Mathematics Subject Classification.} 18A25, 18C35}
\footnotetext{\emph{email:} shenry2@uottawa.ca}
\renewcommand{\thefootnote}{\arabic{footnote}}


\tableofcontents

\section{Introduction}

Given $\kappa$ a regular cardinal and $C$ a category we denote by $\Ind_\kappa(\Ccal)$ the $\kappa$-ind completion of $\Ccal$, that is the pseudo-initial object in the locally full subcategory of $\Ccal \backslash \cat$ whose objects have $\kappa$-filtered colimits and morphisms are functors preserving these $\kappa$-filtered colimits. $\Ind_\kappa(\Ccal)$ can be explicitly described as the full subcategory of the presheaf category $\set^{\Ccal^\op}$ of functors that are small $\kappa$-directed colimits of representables. If $\Ccal$ is small this is also equivalent to the category of $\kappa$-flat functors $\Ccal^\op \to \set$.

The construction $\Ind_\kappa$ is a covariant endofunctor functor of the bicategory of locally small categories. In particular, for each $i \in I$ the evaluation functor $ev_c: \Ccal^I \to \Ccal$, induces a functor preserving $\kappa$-filtered colimits $\Ind_\kappa(C^I) \to \Ind_\kappa(C)$, which together induce a functor:

\[ E^I_{\Ccal,\kappa} : \Ind_\kappa(C^I) \to \Ind_\kappa(C)^I \]

which also preserves $\kappa$-filtered colimits.

The goal of this paper is to investigate under which condition on $C, \kappa$ and $I$ this functor $E^I_{\Ccal, \kappa}$ is an equivalence.

This is also closely related to the question of whether, given an accessible category the category $A^I$ is accessible, and what the locally $\kappa$-presentable objects of this category are:

\begin{prop}\label{prop:Acc_vs_ind} Let $A$ be a $\kappa$-accessible category, with $A_\kappa$ its full subcategory of $\kappa$-presentable objects. and $I$ any category, then the following condition are equivalent:

  \begin{enumerate}
  \item The functor
    \[ E^I_{A_\kappa, \kappa} : \Ind_\kappa(A_\kappa^I) \to \Ind_\kappa(A_\kappa)^I \]
    is an equivalence.

  \item The category of functors $I \to A$ is $\kappa$-accessible, with its $\kappa$-presentable objects being the functors $I \to A_\kappa$.

  \end{enumerate}
  
\end{prop}

\begin{proof}
This follows immediately from the fact that a category $A$ is $\kappa$-accessible if and only if $A = \Ind_\kappa(A_\kappa)$.
\end{proof}

From now on, if $A$ is a $\kappa$-accessible category, we denote by $A_\kappa$ the category of $\kappa$-presentable objects of $A$.

Here it should be noted that a category $A$ is accessible if and only if it is of the form $A \simeq \Ind_\kappa(\Ccal)$ with $\Ccal$ a small category. Moreover, in this case, the $\kappa$-presentable objects of $A$ are exactly the retracts of objects of $\Ccal$. So, a category is of the form $A_\kappa$ for $A$ a $\kappa$-accessible category if and only if it is Cauchy complete. Hence, the question of whether, for any accessible category $A$, the functor category $A^I$ is $\kappa$-accessible with its $\kappa$-presentable objects being the functor $I \to A_\kappa$, is exactly the same as the question of whether $E^I_{\Ccal,\kappa}$ is an equivalence for any Cauchy complete category $\Ccal$.

In \cite{makkai1988strong}, Makkai claims (As theorem 5.1) that for any $\kappa$-accessible categories $A$ and any $\kappa$-small category $I$ then the category of functors $I \to A$ is $\kappa$-accessible, with its $\kappa$-accessible objects being the functors $I \to A_\kappa$. This would imply that $E^I_{\Ccal,\kappa}$ is an equivalence for all Cauchy complete category $\Ccal$ for all $I$ a $\kappa$-small category. We will show that this is not the case - and hence that Makkai's theorem is incorrect. This result is used in a few places throughout the literature, the author is aware of \cite{rogers2021toposes}, \cite{jacqmin2021stability}, \cite{carboni2001syntactic} and \cite{di2020gabriel}. However, in each of these cases, it seems the use of Makkai's theorem can be replaced by one of the two (correct) theorems below:

\begin{theorem}\label{main_theorem_loc_pres} For a category $I$ the following conditions are equivalent:

  \begin{enumerate}
  \item[\namedlabel{main_theorem_loc_pres:acc_cat}{(L1)}] For any locally $\kappa$-presentable category $A$, the category $A^I$ is locally $\kappa$-presentable and its $\kappa$-presentable objects are the functors $I \to A_\kappa$.
  \item[\namedlabel{main_theorem_loc_pres:ind_cat}{(L2)}] For every category $\Ccal$ with all $\kappa$-small colimits, the functor

    \[ E^I_{\Ccal,\kappa} : \Ind_\kappa(C^I) \to \Ind_\kappa(C)^I \]

    is an equivalence.

  \item[\namedlabel{main_theorem_loc_pres:small}{(L3)}] $I$ is essentially $\kappa$-small (that is is equivalent to a $\kappa$-small category).
  \end{enumerate}
  
\end{theorem}

\begin{theorem}\label{main_theorem_acc} For a category $I$ the following conditions are equivalent:

  \begin{enumerate}
      \item[\namedlabel{main_theorem_acc:ind_cat}{(A1)}] \label{main_theorem_acc:ind_cat} For every category $\Ccal$ the functor

        \[ E^I_{\Ccal,\kappa} : \Ind_\kappa(C^I) \to \Ind_\kappa(C)^I \]

 is an equivalence.

      \item[\namedlabel{main_theorem_acc:ind_cat_cauchy}{(A2)}] \label{main_theorem_acc:ind_cat_cauchy} For every Cauchy complete category $\Ccal$, the functor $E^I_{\Ccal,\kappa}$ above is an equivalence.

  \item[\namedlabel{main_theorem_acc:acc_cat}{(A3)}] \label{main_theorem_acc:acc_cat} For any $\kappa$-accessible category $A$, the category $A^I$ is $\kappa$-accessible and its $\kappa$-presentable objects are the functor $I \to A_\kappa$.

  \item[\namedlabel{main_theorem_acc:small}{(A4)}] \label{main_theorem_acc:small} $I$ is essentially $\kappa$-small and well-founded in the sense of \cref{prop:weak_fw}.
  \end{enumerate}
 
\end{theorem}

We refer the reader to \cref{prop:strict_wf} and \cref{prop:weak_fw} for various equivalent definitions of well-founded categories, but one of these characterizations is that $I$ has no non-trivial endomorphisms and that its posetal reflection is a well-founded poset. In particular, in the case of $\kappa = \omega$, condition \ref{main_theorem_acc:small} means that $I$ is equivalent to a finite category with no non-identity endomorphisms. That fact that $\Ind(C^I) = \Ind(C)^I$ for such category was already proved as Proposition 8.8.5 of exposé I of \cite{SGA4I}, as well as in C.Meyer PhD Thesis (page 55) \cite{meyer1983completion}. So in this case, our contribution is only to show that this condition is necessary.

Similarly, Proposition 5.3.5.15 from \cite{lurie2009higher} shows in the framework of $\infty$-categories that $E^I_{\Ccal,\kappa}$ is an equivalence for any regular cardinal $\kappa$ and any $\infty$-category $\Ccal$ when $I$ is a finite poset. This result can be applied as is to $1$-categories, so it does recover a special case of our \cref{main_theorem_acc}, this time beyond the case $\omega = \kappa$, but with less general conditions on the category $I$.

For an explicit counter-example to Makkai's theorem, the reader should go to \cref{subsec:acc_ind_implies_small}, where we show, using an explicit construction, that point \ref{main_theorem_acc:ind_cat_cauchy}, or equivalently \ref{main_theorem_acc:acc_cat}, implies point \ref{main_theorem_acc:small} in \cref{main_theorem_acc}. In particular, for any $\kappa$-small category $I$ which is \emph{not} well-founded, we will build a category $C = I^{(\kappa)}$ (see \cref{cstr:Ialpha}) so that the accessible category $A = \Ind_\kappa(C)$, is such that not every functor in $A^I$ is a $\kappa$-filtered colimit of functors $I \to A_\kappa $ (here $C = A_\kappa$ because $C = I^{(\kappa)}$ will be Cauchy-complete).

It should be noted that the requirement in conditions \ref{main_theorem_acc:acc_cat} and \ref{main_theorem_loc_pres:acc_cat} that the $\kappa$-presentable objects of $A^I$ are the functors $I \to A^\kappa$ is absolutely essential to both theorems. For example, in the case of locally presentable categories, we have

\begin{theorem}\label{th:C^I_always_loc_pres} Let $\Ccal$ be a locally $\kappa$-presentable category, and $I$ be any small category. Then the category of functors $\Ccal^I$ is locally $\kappa$-presentable.
 \end{theorem}

 \begin{proof}
   This follows from Theorem 2.17 of G. Bird PhD thesis \cite{gregorybird}, which claims that the bicategory of locally $\kappa$-presentable category and $\kappa$-accessible right adjoint functors between them has all $\cat$-enriched pseudo limits and they are preserved by the forgetful functor to $\cat$. The functor category $\Ccal^I$ corresponds to the co-tensor for the locally $\kappa$-presentable category $\Ccal$ by the category $I$.
 \end{proof}

 Hence in \cref{main_theorem_loc_pres}, condition \ref{main_theorem_loc_pres:acc_cat} could be rephrased as simply: the $\kappa$-presentable objects of $A^I$ are exactly the functor $I \to A_\kappa$.

 \bigskip

Finally, after the publication of a first preprint version of the present paper. Leonid Positelski published a result that significantly improved some aspect of our \cref{main_theorem_loc_pres} by showing the requirement that the category $A$ is locally presentable can be considerably weakened, without automatically falling under the scope of \cref{main_theorem_acc}. More precisely, theorem 6.1 of \cite{positselski2023notes} assert that:

 \begin{theorem}[Positelski \cite{positselski2023notes}]\label{PositelskiTh} Let $\kappa$ be a regular cardinal and $\lambda < \kappa$ another infinite cardinal. If $I$ is a $\kappa$-small category and $A$ is a $\kappa$-accessible category which has colimits of $\lambda$-indexed chains, then the category $A^I$ is $\kappa$-accessible and its $\kappa$-presentable objects are the functor $I \to A_\kappa$. \end{theorem}

 In particular, using this and \cref{prop:Acc_vs_ind}, we obtain that $\Ind_\kappa (C^I) \simeq \Ind_\kappa(C)^I$ when $I$ is $\kappa$-small and $C$ is Cauchy complete with colimits of $\lambda$-chain for $\lambda < \kappa$ an infinite cardinal.
 
 Note that \cite{positselski2023notes} also proves similar results for more general weighted limits of accessible categories. This result also immediately gives a very good upper bound on the accessibility rank of $A^I$ in general:

 \begin{cor}[Positelski] Let $\kappa$ be a regular cardinal, $I$ a $\kappa$-small category, $A$ a $\kappa$-accessible category and $\lambda$ any regular cardinal such that $\kappa \triangleleft \lambda$. Then $A^I$ is $\lambda$-accessible and its $\lambda$-presentable objects are the functors $I \to A_\lambda$.
 \end{cor}

 Where $\kappa \triangleleft \lambda$ is the ``sharply less'' relation from \cite[Definition 2.12]{adamek1994locally}. This applies for example of $\lambda = \kappa^+$ is the succesor cardinal of $\kappa$.

 \begin{proof}
   Under the condition $\kappa \triangleleft \lambda$, the category $A$ is also $\lambda$-accessible and has $\kappa$-directed colimits. In particular it has colimits of chain indexed by $\kappa$ for $\kappa < \lambda$ an infinite cardinal, so we can apply \cref{PositelskiTh} and concludes.
 \end{proof}

This paper arose following a discussion on Mathoverflow \cite{442055}. In particular, I am especially grateful to Ben Wieland for suggesting a first counter example to the claim that $E^I_{\omega,\Ccal}$ is an equivalence when $I$ is $\omega$-small, which was the starting point to the proof in subsection \ref{subsec:acc_ind_implies_small}, and to Ivan Di Liberti for pointing me to Makkai's theorem 5.1 in \cite{makkai1988strong}.
 
 \section{Proof of \cref{main_theorem_loc_pres}.}

 The equivalence of conditions \ref{main_theorem_loc_pres:acc_cat} and \ref{main_theorem_loc_pres:ind_cat} of \cref{main_theorem_loc_pres} follows immediately from \cref{prop:Acc_vs_ind} and the fact that a $\kappa$-accessible category is locally presentable if and only if its $\kappa$-presentable objects have $\kappa$-small colimits. So we only need to show the equivalence with condition \ref{main_theorem_loc_pres:small}.

We start by observing the following equivalences:
 
 \begin{prop}\label{prop:accessible_ends} Let $I$ be a category. The following conditions are equivalents:

   \begin{enumerate}
   \item \label{prop:accessible_ends:Hom_pres} The functor
     \[ \Hom(\_,\_) : I^\op \times I \to \set \]
     is a $\kappa$-presentable object of $\set^{I^\op \times I}$.
   \item \label{prop:accessible_ends:ends} $I$-indexed ends, seen as functors:

     \[ \int_I  : \set^{I^\op \times I} \to \set \]

     preserves $\kappa$-filtered colimits.
     
   \item \label{prop:accessible_ends:arbitrary_cat} For any category $A$ with $\kappa$-filtered colimits, a functor $I \to A_\kappa$ is $\kappa$-presentable when seen as an object of $A^I$.
   \item \label{prop:accessible_ends:loc_pres_cat} For any locally $\kappa$-presentable category $A$, a functor $I \to A_\kappa$ is $\kappa$-presentable when seen as an object of $A^I$.
  
  \end{enumerate}

Moreover, all these condition holds when $I$ is an essentially $\kappa$-small category.
  
 \end{prop}

Note that, at least in the case $\kappa= \omega$, the conditions of the proposition are much weaker than $I$ being $\omega$-small, that is finite. For example, any finitely generated category can be shown to satisfies these conditions. I do not know if for $\kappa$ uncountable there are such example of non-$\kappa$-small categories satisfying these conditions. We refer to \cite{loregian2015coend} for general material about ends.

\begin{proof} The equivalence of conditions \ref{prop:accessible_ends:Hom_pres} and \ref{prop:accessible_ends:ends} is immediate because of the natural isomorphism:

  \[ \int_I A(x,x) \simeq \Nat(\Hom(x,y), A(x,y))\]

  And the fact that $\Hom$ being $\kappa$-presentable means that the functor $\Nat(\Hom(x,y), \_)$ preserve $\kappa$-filtered colimits. Condition \ref{prop:accessible_ends:ends} implies \ref{prop:accessible_ends:arbitrary_cat} because of the expression of the morphism in the category of functors $A^I$ as:

\[ \Hom_{A^I}(F,G) \simeq \int_{i \in I} \Hom_A (F(i),G(i)) \]

Hence given any filtered diagram $(G_j)_{j \in J}$ and $F: I \to A_\kappa$ a functor, we have an isomorphism

\[
  \begin{array}{ccc}
    \Hom_{A^I}(F,\colim_j G_j ) &\simeq& \int_{i \in I} \Hom_A (F(i),\colim_j G_j(i)) \\ &\simeq& \int_{i \in I} \colim_j \Hom_A (F(i),G_j(i))\\
                              &\simeq& \colim_j  \int_{i \in I}\Hom_A (F(i),G_j(i))\\
    &\simeq& \colim_j \Hom_{A^I}(F,G_j ) \\
  \end{array}
\]

showing that $F$ is indeed $\kappa$-presentable.

The implication $\ref{prop:accessible_ends:arbitrary_cat} \Rightarrow \ref{prop:accessible_ends:loc_pres_cat}$ is tautological, and finally $\ref{prop:accessible_ends:loc_pres_cat} \Rightarrow \ref{prop:accessible_ends:Hom_pres}$ follows from the identification
\[ \Fun( I , \set^{I^\op}) \simeq \Fun( I^\op \times I ,\set).\]

The category $\set^{I^\op}$ is locally $\kappa$-presentable, with the representable object being $\kappa$-presentable (this holds for any $\kappa$), hence by condition \ref{prop:accessible_ends:loc_pres_cat}, the Yoneda embeddings $I \to \set^{I^\op}$ is a $\kappa$-presentable object of this functor category, and through the equivalence above, this corresponds to the functor $\Hom : I^\op \times I \to \set$, hence concluding the proof.

Finally, if $I$ is a $\kappa$-small category, then the end involved in \ref{prop:accessible_ends:ends} can be rewritten as a limit indexed by the twisted arrow category of $I$, which is a $\kappa$-small limits, and hence it preserves $\kappa$-filtered colimits. \end{proof}

 We also mention the following corrolary of \cref{prop:accessible_ends}, which will be usefull in the proof of \cref{main_theorem_acc} later, and is also interesting in its own right. This is directly inspired by proposition 8.8.2 of \cite{SGA4I}.

\begin{cor} \label{prop:E_fullyfaithful} Let $I$ be an essentially $\kappa$-small category, or more generally a category satisfying the equivalent conditions of \cref{prop:accessible_ends}, then the functor
  \[ E^I_{\Ccal,\kappa} : \Ind_\kappa(C^I) \to \Ind_\kappa(C)^I \]
is fully faithful.
 \end{cor}

 \begin{proof}

   Let $X$ and $Y$ be two objects of $\Ind_\kappa(C^I)$, we write them as $\kappa$-directed colimits, $X = \colim X_i$ and $Y=\colim Y_j$  of diagrams in $C^I$. In the category $\Ind_\kappa(C^I)$ we have
   \[
     \begin{array}{rcl}
       \Hom(X,Y) &=& \displaystyle \Hom(\colim_i X_i, \colim_j Y_j) \\
                 &=& \displaystyle \Lim_i \Hom(X_i, \colim Y_j) \\
       &=& \displaystyle \Lim_i \colim_j \Hom(X_i, Y_j)
     \end{array}
   \]
   as the $X_i$ are $\kappa$-presentable in $\Ind_\kappa(C^I)$. In the category $\Ind_\kappa(C)^I$ we have
   \[
     \begin{array}{rcl}
       \Hom(E^I_{\Ccal,\kappa} (X),E^I_{\Ccal,\kappa}(Y)) &=& \displaystyle \Hom(\colim_i E^I_{\Ccal,\kappa} (X_i), \colim_j Y_j) \\
                                                   &=& \displaystyle \Lim_i \Hom( X_i , \colim Y_j ) \\
                                                   &=& \displaystyle \Lim_i \colim_j \Hom( X_i , Y_j )          
     \end{array}
      \]
where we have used that the functor $E$ preserves $\kappa$-directed colimits by construction, and that by \cref{prop:accessible_ends} the $X_i \in C^I$ are $\kappa$-presentable objects in $\Ind_\kappa(C)^I$. This concludes the proof as one easily see by functoriality of the isomorphisms above that the identification $\Hom(X,Y) =  \Hom(E^I_{\Ccal,\kappa} (X),E^I_{\Ccal,\kappa}(Y))$ we obtained is induced by the action of $E^I_{\Ccal,\kappa}$.
 \end{proof}
 
 \subsection{Proof of \ref{main_theorem_loc_pres:acc_cat} or  \ref{main_theorem_loc_pres:ind_cat} $\Rightarrow$ \ref{main_theorem_loc_pres:small}}

 We fix $I$ for which the equivalent conditions \ref{main_theorem_loc_pres:acc_cat} and \ref{main_theorem_loc_pres:ind_cat} of \cref{main_theorem_loc_pres} holds. We will show that $I$ is $\kappa$-small. We first have

 \begin{lemma}\label{Lemma:I_loc_small}
    Any category $I$ satisfying conditions \ref{main_theorem_loc_pres:acc_cat}  or \ref{main_theorem_loc_pres:ind_cat} of \cref{main_theorem_loc_pres} is locally $\kappa$-small, that is has $\kappa$-small Hom sets.
  \end{lemma}

  \begin{proof}
    We apply condition \ref{main_theorem_loc_pres:acc_cat} to the category $\set$, whose $\kappa$-presentable objects are the $\kappa$-small sets. It follows that for every $x \in I $ the representable functor
    \[
      \begin{array}{ccc}
        I &\to& \set \\
        y & \mapsto & \Hom(x,y)
        \end{array}
      \]
      can be written as a $\kappa$-filtered colimit of functors $I \to \set_\kappa$. In particular, there exists a functor $A:I \to \set_\kappa$ and a natural transformation $\lambda_y:A(y) \to \Hom(x,y)$, such that the identity functor $x \to x$ can be writen as $\lambda_x(e)$ for $e \in A(x)$. But it then follows that for every arrow $p:x \to y$, we have $\lambda_y(p e) = p \lambda_y(e) = p \circ \id_x = p $, hence $A(y) \to \Hom(x,y)$ is surjective, and hence $\Hom(x,y)$ is a $\kappa$-small set for all $x,y \in I$.
   \end{proof}

   We can now conclude the proof of this implication with:

   \begin{lemma}
     Any category $I$ satisfying conditions \ref{main_theorem_loc_pres:acc_cat} or \ref{main_theorem_loc_pres:ind_cat} of \cref{main_theorem_loc_pres} is essentially $\kappa$-small.
   \end{lemma}

   \begin{proof} We have seen in \cref{Lemma:I_loc_small} that $I$ is locally $\kappa$-small. $I$ also satisfies the last (hence all) condition of \cref{prop:accessible_ends}. Hence the functor $\Hom : I^\op \times I \to \set$ is $\kappa$-presentable. In general, given a $\kappa$-presentable object $X$ of a functor category $\set^K$, one can show that there is a $\kappa$-small family of elements $a_x \in X(k_x)$ such that every element of $X$ is the image of one of $a_x$ by the functoriality of $X$. Indeed each such family defines a subobject of $X$ and together they form a $\kappa$-filtered family of subobjects of $X$, so if $X$ is $\kappa$-presentable, then one of these subobjects is equal to $X$.

   In our case, it means that there exists a $\kappa$-small set of arrows $f_x:a_x \to b_x \in I$ for $x \in X$ such that every arrow $g$ of $I$ can be factored through one of these as $g =u f_x v$ for some $x \in X$. In particular, for each object $y \in I$, we have two arrows $u,v$ such that $Id_y = u f_x v$, which implies that $y$ is a retract of $a_x$ (as well as of $b_x$). The category $I$ being locally $\kappa$-small, the full subcategory of the $a_x$ is a $\kappa$-small category $A$ and we just showed that $I$ identifies with a full subcategory of the Cauchy completion of $A$, hence is an essentially $\kappa$-small category, as the Cauchy completion of a $\kappa$-small category can be constructed as a $\kappa$-small category. \end{proof}

 \subsection{ Proof of \ref{main_theorem_loc_pres:small} $\Rightarrow$  \ref{main_theorem_loc_pres:acc_cat} }

 We fix $A$ a locally $\kappa$-presentable category and $I$ a $\kappa$-small category. We will show condition \ref{main_theorem_loc_pres:acc_cat}, i.e. that $A^I$ is also locally $\kappa$-presentable with its $\kappa$-presentable objects being the functors taking values in the full subcategory $A_\kappa$ of $\kappa$-presentable objects of $A$. 
 Note that by \cref{prop:accessible_ends}, as $I$ is $\kappa$-small, the functors $I \to A_\kappa$ are indeed $\kappa$-presentable objects of $A^I$.

 The evaluation functor $ev_i: A^I \to A$ (for $i \in I$) have left adjoints $F_i : A \to A^I$ than can be expressed as
 \[F_i (X) \coloneqq \left( j \mapsto \coprod_{\Hom_I(i,j)} X \right) \in A^I.\]

 In particular, as the category $I$ is $\kappa$-small this coproduct is $\kappa$-small and hence if $X \in A_\kappa$, then $F_i(X) \in (A_\kappa)^I$. We have that for any $U \in A^I$, $\Hom(F_i(X),U) = \Hom(X,ev_i(U))$, so it follows that an arrow $f: U \to V$ in $A^I$ is an isomorphism if and only if for each $X \in A_\kappa$ and each $i \in I$ we have that
 \[\Hom(F_i(X),U) \to \Hom(F_i(X),V)\]
 is an equivalence. The following lemma, applied to the cocomplete category $\Acal^I$ and to $\Ccal = (\Acal_\kappa)^I$  then concludes the proof:

 \begin{lemma}
   Let $\Acal$ be a cocomplete category and let $\Ccal \subset \Acal$ be a full subcategory of $\Acal$ such that:

   \begin{enumerate}
   \item $\Ccal$ is closed under $\kappa$-small colimits in $\Acal$.
    \item Every object of $\Ccal$ is $\kappa$-presentable in $\Acal$.
    \item For any arrow $f :U \to V$, if for all $c \in \Ccal, \Hom(c,f):\Hom(c,U) \to \Hom(c,V)$ is a bijection, then $f$ is an isomorphism.
    \end{enumerate}

    Then, $\Acal$ is locally $\kappa$-presentable and up to equivalence, $\Ccal$ is the category of $\kappa$-presentable objects of $\Acal$.
 \end{lemma}

 \begin{proof}
   This is essentially the definition of a locally presentable categories, depending on the reference. We just briefly recall the argument: for any object $X \in \Acal$, we let
   \[ Y  = \colim_{c \to X \atop c \in \Ccal} c, \]
 as $\Ccal$ has all $\kappa$-small colimits, this is a $\kappa$-filtered colimit. As every object $d \in \Ccal$ is $\kappa$-presentable we have that $\displaystyle \Hom(d,Y) = \colim_{c \to X} \Hom(d,c) = \Hom(d,X)$, hence the last condition implies that the canonical map $Y \to X$ is an isomorphism. So, $\Ccal$ is a dense subcategory of $\kappa$-presentable objects, hence $\Acal$ is locally $\kappa$-presentable. Finally, if $X$ is a $\kappa$-presentable object then as $X$ is a $\kappa$-directed colimits of objects of $\Ccal$, then $X$ is a retract of an object in $\Ccal$, and as $\Ccal$ has all $\kappa$-small colimits, it is closed under retracts, so that $X$ is isomorphic to an object of $\Ccal$. 
 \end{proof}

\section{Proof of \cref{main_theorem_acc}.}

The equivalence between condition \ref{main_theorem_acc:ind_cat_cauchy} and condition \ref{main_theorem_acc:acc_cat} of \cref{main_theorem_acc} follows immediately from \cref{prop:Acc_vs_ind} and the remarks right after its proof. The implication \ref{main_theorem_acc:ind_cat} $\Rightarrow$ \ref{main_theorem_acc:ind_cat_cauchy} is tautological, so we only need to show $\ref{main_theorem_acc:ind_cat_cauchy} \Rightarrow \ref{main_theorem_acc:small}$ and $\ref{main_theorem_acc:small} \Rightarrow \ref{main_theorem_acc:ind_cat}$. But before this, we need to discuss the notion of well-founded categories which appear in condition \ref{main_theorem_acc:small}.

\subsection{Well-founded categories}

  The class $\Ord$ of all ordinal is seen as a (large) category with a single arrow from $\beta \to \gamma$ if $\gamma \leqslant \beta$. Any ordinal $\alpha$ is seen as the small full subcategory $\alpha \subset \Ord$ of all ordinals $\beta < \alpha$. 

We first need to introduce the following construction, which plays a central role both in the notion of well-founded categories and latter in the proof of \cref{main_theorem_acc}.
  
\begin{construction}\label{cstr:Ialpha}
  Given $I$ a category and $\alpha$ either an ordinal or the large category $\Ord$ of all ordinal, we denote by $I^{(\alpha)}$ the non-full subcategory of $I \times \alpha$ which contains all the object of $I \times \alpha$ and in which the morphisms are:

  \begin{enumerate}
  \item All arrows $(x,\beta) \to (y, \gamma)$ in $I \times \alpha$ if $\beta < \gamma$.
  \item Only the identity arrow $(x,\beta) \to (x,\beta)$.
  \end{enumerate}

  The projection $I \times \alpha \to I$ restrict to a functor $I^{(\alpha)} \to I$ which we call the canonical functor.
\end{construction}

It should be noted that the construction $I \mapsto I^{(\alpha)}$ is not functorial in the bicategorical sense, but only in a 1-categorical sense, as it explicitely involve the set of objects of $I$. This construction does not respect the ``equivalence principle'' in the sense that an equivalence of category $I \simeq J$ does not imply that $I^{(\alpha)} \simeq J^{(\alpha)}$.

\bigskip

A binary relation $R$ on a set $X$ is said to be well-founded if there is no infinite chain $x_1, \dots , x_n, \dots$ in $X$ such that $x_{n+1} R x_n$ for all $n$. Equivalently, if the only subset $S \subset X$ satisfying $(\forall y, y R x \Rightarrow y \in S) \Rightarrow x \in S$ is $S=X$. A poset is said to be well-founded if the relation $ < $ defined as $x \leqslant y$ and $x \neq y$ is well-founded. For example ordinals are well-founded as posets, and up to isomorphisms they are the unique well-founded totally ordered sets.

\bigskip

A functor $F : \Ccal \to \Dcal$ is said to be \emph{identity-reflecting} if for every arrow $f$, $F(f)$ is an identity arrow implies that $f$ is an identity arrow. Note that this notion also breaks the equivalence principle: a functor equivalent to an identity-reflecting functor doesn't have to be identity-reflecting.

\bigskip

The posetal reflection of a category $I$, is the universal poset with a functor from $I$. One start with the relation on the set of objects of $I$ defined by $x \leqslant y \coloneqq$ ``There exists an arrow $x \to y$'' which is transitive and reflective and then one quotient the set of objects by the equivalence relation $x \leqslant y$ and $y \leqslant x$ to make into a poset. 

\begin{lemma}\label{poset=no_endo} For a category $I$ the following conditions are equivalent:

  \begin{enumerate}
  \item\label{poset=no_endo:id_reflec} The functor from $I$ to its posetal reflection is identity-reflecting.
  \item\label{poset=no_endo:no_endo} The category $I$ has no non-identity isomorphisms or endomorphisms.
  \end{enumerate}
\end{lemma}

\begin{proof}
  Any isomorphism or endomorphism is sent to an identity in the posetal reflection of $I$, so the implication \ref{poset=no_endo:id_reflec} $\Rightarrow$ \ref{poset=no_endo:no_endo} is clear. We hence assume that $I$ has no non-identity endomorphisms or isomorphisms. Two objects $x,y$ of $I$ become identified in the posetal reflection of $I$ if and only if there are maps $f:x \to y$ and $g: y \to x$, but then the composite $f \circ g$ and $g \circ f$ are endomorphisms, hence identity, hence $f$ and $g$ are isomorphisms, and hence $x = y$. It follows that the map from $I$ to its posetal reflection is bijective on objects, and as $I$ has no non-identity endomorphisms this makes it identity-reflecting.
\end{proof}

\begin{prop}\label{prop:strict_wf} For a category $I$, the following conditions are equivalents:

  \begin{enumerate}
        \item[\namedlabel{prop:strict_wf:chains}{(SW1)}] There are no identity-reflecting functors $\omega^\op \to I$.
  \item[\namedlabel{prop:strict_wf:relation}{(SW2)}] The relation $x < y$ on objects of $I$ defined by ``there exists a non-identity arrow $x \to y$'' is well-founded.
    \item[\namedlabel{prop:strict_wf:no_endo}{(SW3)}] The category $I$ has no non-identity isomorphisms or endomorphisms and its posetal reflection is a well-founded poset. 
    \item[\namedlabel{prop:strict_wf:fct_to_ord}{(SW4)}] There is an identity-reflecting functor $\Ccal \to \Ord$.
   \item[\namedlabel{prop:strict_wf:retract}{(SW5)}] The canonical functor $I^{(\Ord)} \to I$ admits a section (up to equality)
  \end{enumerate}
  A category satisfying these conditions is said to be \emph{strictly well-founded}.
\end{prop}

\begin{proof}
  The equivalence between \ref{prop:strict_wf:chains} and \ref{prop:strict_wf:relation} is immediate as the functors mentioned in \ref{prop:strict_wf:chains} are exactly the downward chains for the relation mentioned in \ref{prop:strict_wf:relation}.

  \ref{prop:strict_wf:relation} $\Rightarrow$ \ref{prop:strict_wf:no_endo} : indeed, any isomorphisms or endomorphisms in $I$ would allow to obtain either a $x$ such that $x < x$ or $x,y$ such that $x <y $ and $y < x$ which is impossible in a well-founded relation, and if there are no isomorphisms or endomorphisms, then the posetal reflection is the set of objects with the relation of the point \ref{prop:strict_wf:relation}.

  \ref{prop:strict_wf:no_endo} $\Rightarrow$ \ref{prop:strict_wf:fct_to_ord} Every well-founded-poset admits a functor to $\Ord$ which is identity-reflecting ( e.g. defined by well-founded induction as $v(x) = \sup_{y<x} v(y)^+$) so the implication follows by \cref{poset=no_endo}.

  \ref{prop:strict_wf:fct_to_ord} $\Rightarrow$ \ref{prop:strict_wf:retract}. Given $F: I \to \Ord$ an identity-reflecting functor, then the functor $(Id,F) : I \to I \times \Ord$ is a section of the first projection and takes values in $I^{(\Ord)}$.

  \ref{prop:strict_wf:retract} $\Rightarrow$ \ref{prop:strict_wf:chains}: a section of $I^{(\Ord)} \to I$ is automatically identity reflecting, so the existence of such section implies that there is an identity reflecting functor $I \to \Ord$ which clearly contradicts the existence of an identity reflecting functor $\omega^\op \to I$ as there is no such functor $\omega^\op \to \Ord$.

\end{proof}

\begin{prop}\label{prop:weak_fw} For a category $I$, the following conditions are equivalents
  \begin{enumerate}
  \item[\namedlabel{prop:weak_fw:conservatif}{(W1)}] $I$ has no non-identity endomorphisms and it admits a conservative functor to $\Ord$.
  \item[\namedlabel{prop:weak_fw:no_endo}{(W2)}] $I$ has no non-identity endomorphisms and its posetal reflection is well-founded.
  \item[\namedlabel{prop:weak_fw:skeleton}{(W3)}] Every skeleton of $I$ is a strictly well-founded category.
  \item[\namedlabel{prop:weak_fw:equiv}{(W4)}] $I$ is equivalent to a strictly well-founded category.
  \item[\namedlabel{prop:weak_fw:section_iso}{(W5)}] The canonical functor $I^{(\alpha)} \to I$ admits a section up to natural isomorphisms.
  \item[\namedlabel{prop:weak_fw:section_retract}{(W6)}] The identity functor on $I$ is a retract of a functor that can be factored as a functor $I \to I^{(\alpha)}$ followed by the canonical functor  $I^{(\alpha)} \to I$.
  \end{enumerate}

A category satisfying these equivalent conditions will be said to be \emph{Well-founded}.
  
\end{prop}

Condition \ref{prop:weak_fw:section_retract} may seem a little strange - the only reason it is here is because this characterization will be used in the next subsection to show the implication $\ref{main_theorem_acc:ind_cat_cauchy} \Rightarrow \ref{main_theorem_acc:small}$ of \cref{main_theorem_acc}.

\begin{proof}

  \ref{prop:weak_fw:conservatif} $\Rightarrow$ \ref{prop:weak_fw:no_endo}. Such a conservatif functor factors into a conservatif functor fron the posetal reflection of $I$ to $\Ord$, which implies that this posetal reflection has no infinite strictly decreasing chains, hence is well-founded.
  
  \ref{prop:weak_fw:no_endo} $\Rightarrow$ \ref{prop:weak_fw:skeleton}. This follows immediately from point \ref{prop:strict_wf:no_endo} of \cref{prop:strict_wf}: indeed in a skeleton all isomorphisms will be endomorphisms, and hence a skeleton of a category satisfying \ref{prop:weak_fw:no_endo}, will have non-identity endomorphisms and isomorphisms and a well-founded posetal reflection, so satisfy condition \ref{prop:strict_wf:no_endo} of \cref{prop:strict_wf}.
  
  \ref{prop:weak_fw:skeleton} $\Rightarrow$ \ref{prop:weak_fw:equiv} is tautological.

  \ref{prop:weak_fw:equiv} $\Rightarrow$ \ref{prop:weak_fw:section_iso}. The construction $I^{(\Ord)}$ is not a functor in the $2$-categorical or bicategorical sense, but it is functorial in the $1$-categorical sense nonetheless. So given an equivalence $F: A \to I$ with $A$ a strictly well-founded category, we get a commutative square:
\[  \begin{tikzcd}
    A^{(\Ord)} \ar[r,"F^{(\Ord)}"] \ar[d,"\pi_A"] & I^{(\Ord)}  \ar[d,"\pi_I"] \\
    A \ar[r,"F"] & I
  \end{tikzcd}\]

By point \ref{prop:strict_wf:retract} of \cref{prop:strict_wf}, the left functor $\pi_A$ has a section $s$ (up to equality) and the bottom functor $F$ is an equivalence (so it has an inverse up to ismophisms), so composing $F^{(\Ord)} s F^{-1}$ gives a functor $I \to I^{(\alpha)}$ such that if one post-compose it by $\pi_I$ we get $\pi_I F^{(\Ord)} s F^{-1} = F \pi_A s F^{-1} = F F^{-1} \simeq \id_I$ hence the result.

  \ref{prop:weak_fw:section_iso} $\Rightarrow$ \ref{prop:weak_fw:section_retract} is tautological.

  \ref{prop:weak_fw:section_retract} $\Rightarrow$  \ref{prop:weak_fw:conservatif}. We get a functor $F:I \to \Ord$ by composing the functor $I \to I^{(\Ord)}$ with the projection $I^{(\Ord)} \subset I \times \Ord \to \Ord$. Let $f$ be any arrow such that $F(f)$ is an identity. As the only arrows in $I^{(\Ord)}$ sent to identities in $\Ord$ are identities, it follows that the image of $f$ is already an identity arrow in $I^{\Ord}$, hence $f$ is a retract of an identity arrow in $I$, so it has to be an isomorphism. This proves that the functor to $\Ord$ is conservative. If we further assume that $f$ is an endomorphism of an object, then the same argument shows that $f$ is a retract of an identity, with the same retraction on each side, which forces $f$ to be an identity arrow, hence this concludes the proof.
\end{proof}

\subsection{Proof of \ref{main_theorem_acc:ind_cat_cauchy} $\Rightarrow$ \ref{main_theorem_acc:small} }
\label{subsec:acc_ind_implies_small}

We fix $I$ a category such that for all Cauchy complete category $\Ccal$, the functor $E^I_{\Ccal,\kappa} : \Ind_{\kappa}(\Ccal^I) \to \Ind_\kappa(\Ccal)^I$ is an equivalence. It is in particular an equivalence for all category $\Ccal$ having $\kappa$-small colimits, so by \cref{main_theorem_loc_pres} the category $I$ is $\kappa$-small. 

We then take $\Ccal = I^{(\kappa)}$. For each $x \in I$, we consider the object $E_x \in \Ind_\kappa$ defined as follows:

\[ E_x = \colim_{\alpha < \kappa} (x,\alpha)\]

As $\kappa$ is assumed to be a regular cardinal (which we consider as an ordinal here), the poset $\kappa$ has all $\kappa$-small join and hence is $\kappa$-directed. As a functor $\Ccal^\op \to \set$, $E_x$ can be described as:

\[ E_x( y,\alpha) = \Hom_I (y,x) \]

So this clearly constitutes a functor $E: I \to \Ind_\kappa(\Ccal)$. It should also be noted that the functor $\Ind_\kappa(\pi_I) : \Ind_\kappa(I^{(\alpha)}) \to \Ind_\kappa(I)$ sends the objects $E_x$ to the object $x$ itself as the all the $(x, \alpha)$ are sent to $x$ and hence the colimit defining $E_x$ becomes trivial in $\Ind_\kappa(I)$. So that the composite $\Ind_\kappa(\pi_I) \circ E : I \to \Ind_\kappa(I)$ identifies with the canonical functor $I \to \Ind_\kappa(I)$.

As we are assuming condition \ref{main_theorem_acc:ind_cat_cauchy} of \cref{main_theorem_acc} and the category $\Ccal = I^{(\kappa)}$ is Cauchy complete (it has no non-identity idempotent), we can hence find a $\kappa$-directed family of functors $F^j: I \to I^{(\kappa)}$ such that $E = \colim_j F^j$ in the category of functors $I \to \Ind_\kappa(I^{(\kappa)})$.

$\Ind_\kappa(\pi_I)$ preserves $\kappa$-filtered colimit, so we also have that

\[ \colim_j \Ind_\kappa(\pi_I) F^j \simeq \Ind_\kappa(\pi_I) E \]

Identify with the canonical functor $I \to \Ind_\kappa(I)$. Now, applying our assumption \ref{main_theorem_acc:ind_cat_cauchy} to (the Cauchy completion of) $I$, we see that this implies that the canonical functor $I \to \Ind_\kappa(I)$ is a $\kappa$-presentable object of the category of all such functor, and hence because of the previous colimit it has to be a retract of one of the functors  $\Ind_\kappa(\pi_I) F^j$, but then all the functors involved actually takes values in $I$ and hence we have shown that the identity of $I$ is a retract of $\pi_I \circ F^j$ for some $j$, which is exactly condition \ref{prop:weak_fw:section_retract} of \cref{prop:weak_fw}. Hence proving that $I$ is well-founded.

\subsection{Proof of \ref{main_theorem_acc:small} $\Rightarrow$ \ref{main_theorem_acc:ind_cat}}

We are now showing that if $I$ is well-founded and $\kappa$-small and $\Ccal$ is any category, then $E^I_{\Ccal,\kappa}: \Ind_\kappa(\Ccal^I) \to \Ind_\kappa(\Ccal)^I$ is an equivalence. The strategy here is to show first that, for $I$ a $\kappa$-small category and $\alpha < \kappa$ an ordinal, the functor
\[ E^{I^{(\alpha)}}_{\Ccal,\kappa}: \Ind_\kappa(\Ccal^{I^{(\alpha)}}) \to \Ind_\kappa(\Ccal)^{I^{(\alpha)}} \]
is an equivalence, which we achieve by induction on $\alpha$, and then we exploit that when $I$ is well-founded it is a retract of one of the $I^{(\alpha)}$ to conclude the proof.

We start with the following proposition:

\begin{prop}\label{prop:tower_limits} Let $\alpha < \kappa$ any $\kappa$-small ordinal. Let $\Ccal_\bullet: \alpha^\op \to \cat$ be a tower of categories with the property that for each $\gamma < \alpha$ the functor

  \[ \Ccal_\gamma \to \Lim_{\beta < \gamma } \Ccal_\beta \]

  is (equivalent to a) cartesian fibration. Then the limit $\Lim_{\beta < \alpha} \Ccal_\beta$ is preserved by $\Ind_\kappa$.  
  
\end{prop}

\begin{remark}
Here by limits, we mean pseudo-limits. As the $\Ind_\kappa$ functor is only well defined up to equivalence, asking for the preservation of strict limits does not really make sense. Because of this, it does not make sense either to ask the comparison functors in the proposition to be Grothendieck cartesian fibration in the strict sense, as they are only well defined up to equivalences of categories. This is why we only require that they are equivalent to cartesian fibration (equivalently, are Street fibrations). Of course, one could take all limits to be strict limits, and then one could ask these functors to be Grothendieck fibrations. As Grothendieck fibrations are in particular isofibrations, these strict limits would be equivalent to the corresponding pseudo-limits. The $\Ind_\kappa$ functor would then preserves the strict limit up to equivalences of categories. 
\end{remark}

\begin{proof}
  We fix $\alpha$ a $\kappa$-small ordinal and

  \[ \Ccal_0 \leftarrow \Ccal_1 \leftarrow \dots \leftarrow \Ccal_\gamma \leftarrow \dots \]
 a sequence of categories indexed by $\alpha^\op$, whose transition maps are cartesians fibrations. We need to show that the inclusion
  \[ \Lim_{\gamma \in \alpha^\op} \Ccal_\gamma \subset  \Lim_{\gamma \in \alpha^\op} \Ind_\kappa(\Ccal_\gamma) \]
  identifies the right-hand side with the $\Ind_\kappa$ completion of the left-hand side. The proof has three parts: first oen shows that the objects of $\Lim_{\gamma \in \alpha^\op} \Ccal_\gamma$ are $\kappa$-presentable in the right hand side, mostly using the same sort of argument as in \cref{prop:accessible_ends}, the second step is to show that the functor
  \[ E: \Ind_\kappa \Lim_{\gamma \in \alpha^\op} \Ccal_\gamma \to  \Lim_{\gamma \in \alpha^\op} \Ind_\kappa(\Ccal_\gamma) \]
  is fully faithful, using the exact same argument as in \cref{prop:E_fullyfaithful}, and finally the third step is to show that this functor is essentially surjective, that is that every object of $\displaystyle \Lim_{\gamma \in \alpha^\op} \Ind_\kappa(\Ccal_\gamma)$  is a $\kappa$-filtered colimits of objects of $\displaystyle \Lim_{\gamma \in \alpha^\op} \Ccal_\gamma$. Here the argument is to show that for all $Y$ in the limits, the diagram of all the $X \to Y$ with $\displaystyle X \in \Lim_{\gamma \in \alpha^\op} \Ccal_\gamma$ is $\kappa$-filtered and has colimit $Y$.

For the first part, we observe that in the limit $\Lim_{\gamma \in \alpha^\op} \Ind_\kappa(\Ccal_\gamma)$, all the transition functors preserve $\kappa$-filtered colimits, so all $\kappa$-filtered colimits are computed componentwise.  The Hom set in the limits can be written as a $\kappa$-small limit 
  \[ \Hom(X,Y) = \Lim_{\gamma \in \alpha^\op} \Hom(X_\gamma,Y_\gamma). \]
So, if for all $\gamma$, the objects $X_\gamma$ is in $\Ccal_\gamma$, and hence $\kappa$-presentable, then each individual Hom functor preserves $\kappa$-filtered colimits in the second variable, and the limits being $\kappa$-small, it comutes to $\kappa$-filtered colimits, hence $\Hom(X,\_)$ preserves $\kappa$-filtered colimits. So that $(X_\gamma) \in \Lim_{\gamma \in \alpha^\op} \Ind_\kappa(\Ccal_\gamma)$ is $\kappa$-presentable.

For the second part, we can just run the exact same argument as in \cref{prop:E_fullyfaithful}. The functor $E$ preserves $\kappa$-filtered colimits by construction, and so we can do the exact same computation as in the proof of \cref{prop:E_fullyfaithful} to conclude that the functor $E$ is fully faithful.

Moving to the third part, we show that for any
\[ Y = (Y_\gamma)_{\gamma \in \alpha^\op} \in \Lim_{\gamma \in \alpha^\op} \Ind_\kappa(\Ccal_\gamma) \]
the category of $X \to Y$ with $X_\gamma \in \Ccal_\gamma$ is $\kappa$-filtered. So let $X^{(i)}$ be a $\kappa$-small diagram of such objects. We construct a cocone for it, that is a factorization $X^{(i)} \to E \to Y$ where all $E_\gamma \in \Ccal_\gamma$ and the first arrow is natural in $i$. This is done by induction on $\gamma$. Indeed assuming such an $E_\beta$ has been constructed for all $\beta < \gamma$, that is we have our (natural) factorization $X^{(i)} \to E \to Y$ in the category $\Lim_{\beta < \gamma } \Ccal_\beta$. First, as $Y_\gamma \in \Ind_\kappa(\Ccal_\gamma)$, exists an object $E^0_\gamma \in \Ccal_\gamma$ that factors the cocone $X_\gamma^{(i)} \to E^0_\gamma \to Y_\gamma$. The functor $\pi: \Ind_\kappa \Ccal_\gamma \to \Lim_{\beta < \gamma } \Ind_\kappa \Ccal_\beta$ preserves $\kappa$-filtered colimits, so we can further ``enlarge'' $E^0_\gamma$ so that its image $\pi(E^0_\gamma)$ in this limit also factors the already existing map

  \[ X^{(i)} \to E \to \pi(E^0_\gamma) \to Y \]

  while making sure that the composite $X^{(i)} \to E \to \pi(E^0_\gamma)$ identifies with the image under $\pi$ of the the cocone structure $X_\gamma^{(i)} \to E^0_\gamma$. Finally, we construct $E_\gamma$ as a cartesian lift of $E \to \pi(E^0_\gamma)$ to a map $E_\gamma \to E^0_\gamma$, and easily check that $E_\gamma$ has all the properties needed to extend $E$.

Finally, we show that any $Y \in \Lim_{\gamma \in \alpha^\op} \Ind_\kappa(\Ccal_\gamma)$ is indeed the colimits of this $\kappa$-filtered diagram. $\kappa$-filtered colimits being computed componentwise it is enough to check that for each $ V \in \Ccal_\gamma$ and any maps $V \to Y_\gamma$, the map can be factored as $V \to X_\gamma \to Y_\gamma$ where $X \to Y$ is a map in the limits with $X \in \Lim_{\gamma \in \alpha^\op} \Ccal_\gamma$, and that given two such factorizations, they can be equalized by some larger $X' \to Y$. This can be achieved by exactly the same construction as above, by just adding one step: when constructing $E^0_\gamma$, one can make it so that (depending on the case) either the map $X_\gamma \to Y_\gamma$ factors through $E^0_\gamma \to Y_\gamma$ or that the two maps $V \rightrightarrows X_\gamma$ are equalized by $E^0_\gamma$, and then proceed with constructing $E^0_\gamma \to E_\gamma$ in the same way. And this concludes the proof.  
\end{proof}

\begin{lemma}\label{lem:sieve_restr_are_fib}
  Let $C$ be any category and $A \subset B$ be a sieve inclusion. That is $A$ is a full subcategory of $B$ such that for $f:b \to a$ with $a \in A$ we have $b \in B$. Then restriction functor $C^B \to C^A$ is a cartesian fibration. 
\end{lemma}

\begin{proof}
  We omit the details. The central observation is that given $F: B \to C$, $E:A \to C$, and $\lambda: E \to F|_B$ a cartesian lift of $\lambda$ is obtained by considering $F' : B \to C$ to be defined as
  \[ F'(b) = \left\lbrace
      \begin{array}{cc}
        E(b) & \text{ if $b \in A$.}\\
        F(b) & \text{Otherwise.}
        \end{array} \right.
    \]
with the functoriality of $F'$ being given by the functoriality of $E$ and $F$ respectively for the arrows whose source and target are either both in $A$ or both outside of $A$, for the arrows $f:a \to b$ with $a \in A$, and $b \notin A$, by
\[E(a) \overset{\lambda_a}{\to} F(a) \overset{F(f)}{\to} F(b) \]
and as $A$ is a sieve, there are no arrows going in the other direction.
\end{proof}

\begin{prop}\label{prop:Result_for_Ialpha} Let $\Ccal$ be any category, $I$ be a $\kappa$-small category and $\alpha< \kappa$ an ordinal then
  \[E^{I^{(\alpha)}}_{\Ccal, \kappa} : \Ind_\kappa\left(\Ccal^{I^{(\alpha)}}\right) \to \Ind_\kappa\left(\Ccal\right)^{I^{(\alpha)}} \]
  is an equivalence.
  
\end{prop}

\begin{proof}
  We proceed by induction on $\alpha$, that is we assume the result is true for all $\beta < \alpha$. In the case of $\alpha = 0$, the category $I^{(\alpha)}$ is the discrete category on the set $X$ of objects of $I$, which is in particular a $\kappa$-small set. It is then easy to check that in this case the map:

  \[E^{X}_{\Ccal, \kappa} : \Ind_\kappa\left(\Ccal^X\right) \to \Ind_\kappa\left(\Ccal\right)^{X} \]
 is an equivalence, which gives the case $\alpha=0$. 

 If $\alpha = \beta^+$ is a successor ordinal, we show that $E^{I^{(\alpha)}}_{\Ccal, \kappa}$ is an equivalence following a strategy similar to the proof of \cref{prop:tower_limits}. First one can apply \cref{prop:E_fullyfaithful} to show that it is fully faithful. So we only need to show that it is essentially surjective, that is that every object  $Y \in \Ind_\kappa\left(\Ccal\right)^{I^{(\alpha)}}$ is a $\kappa$-directed colimit of objects in $\Ccal^{I^{(\alpha)}}$. For this we will proceed in two steps: we first show that the slice $\Ccal^{I^{(\alpha)}}/Y$ is a $\kappa$-filtered category and then that $Y$ is its colimits. In both cases a key observation is that as the result is assumed to be true for $\beta$, both these claims are true when $\beta$ is replaced by $\alpha$.

 So we consider a $\kappa$-small diagram $X^i \to Y$ in $\Ccal^{I^{(\alpha)}}/Y$ and we will show it admits a cocone. First, by our induction hypothesis, the restriction to $I^{(\beta)}$ has a cocone $X^i|_{I^{(\beta)}} \to E \to Y|_{I^{(\beta)}}$. We only need to extend $E$ to the object of the form $(\alpha,i) \in I^{(\alpha)}$, endowed with maps $E(\alpha,i) \to Y(\alpha,i)$, and all the appropriate maps from the $E(\beta,i) \to E(\alpha,i)$ and maps $X^i(\alpha,i) \to E(\alpha,i)$ such that composites, for example $E(\beta,i) \to E(\alpha,i) \to Y(\alpha,i)$, are the correct maps. This can be summed up as the question of finding a cocone for a certain $\kappa$-small diagram in $\Ccal/Y(\alpha,i)$, hence we can build these objects as $Y(\alpha,i) \in \Ind_\kappa(\Ccal)$.

 Finally, similarly to the proof of \cref{prop:tower_limits}, in order to show that $Y$ is the colimits of $\Ccal^{I^{(\alpha)}}/Y$, it is enough to show that for all $\gamma \leqslant \alpha$ and for each arrow $V \to Y(\gamma,i)$ for $V \in \Ccal$, this arrow can be factored as $V \to X(\gamma,i) \to Y(\gamma,i)$ for $X \in \Ccal^I/Y$, and that any two such factorizations are equalized by some $X \to X'$ in $\Ccal^I/Y$. But this is easily done by the exact same argument: One first builds the restriction of $X$ to $I^{(\beta)}$ by the induction hypothesis and then we extend $X$ to $I^{(\alpha)}$ by finding certain cocones for $\kappa$-small diagrams in $\Ccal/Y(\alpha,i)$.








We now move to that last part of the proof: $\alpha$ is a limit ordinal, then $I^{(\alpha)}$ is the union of the $I^{(\beta)}$ for $\beta \subset \alpha$, which are all sieve in $I^{(\alpha)}$. Hence

  \[ \Ccal^{I^{(\alpha)}} = \Lim_{\beta < \alpha} \Ccal^{I^{(\beta)}} \]

  and \cref{lem:sieve_restr_are_fib} immediately implies that this limit satisfies the conditions of \cref{prop:tower_limits}, hence:
  \[ \Ind_\kappa( \Ccal^{I^{(\alpha)}})  \simeq \Lim_{\beta < \alpha} \Ind_\kappa(\Ccal^{I^{(\beta)}}) \]
  hence by our induction hypothesis, we obtain
  \[ \Ind_\kappa( \Ccal^{I^{(\alpha)}})  \simeq \Lim_{\beta < \alpha} \Ind_\kappa(\Ccal)^{I^{(\beta)}} \simeq \Ind_\kappa( \Ccal)^{I^{(\alpha)}}, \]
  which concludes the proof. \end{proof}

We can now prove the claimed implication:

\begin{prop}
  Let $I$ be an essentially $\kappa$-small well-founded category, and $\Ccal$ any category, then
  \[E^{I}_{\Ccal, \kappa} : \Ind_\kappa(\Ccal^{I}) \to \Ind_\kappa(\Ccal)^{I} \]
is an equivalence of categories.
  
\end{prop}

\begin{proof} One can freely assume that $I$ is $\kappa$-small. As $I$ is well-founded, then the projection $I^{(\Ord)} \to I$ admits a section up to isomorphism. The composite functor $ I \to I^{(\Ord)}  \to \Ord$ has a $\kappa$-small image, so it factors through an order preserving inclusion $\alpha \subset \Ord$ for $\alpha$ a $\kappa$-small ordinal.

  The full subcategory of objects of $I^{(\Ord)}$ whose image in $\Ord$ is in this $\kappa$-small ordinal identifies to $I^{(\alpha)}$, and hence we have a section (up to isomorphic) of the projection $I^{(\alpha)} \to I$.

  It follows that the functor $E^I_{\Ccal,\kappa}$ is a retract (up to natural isomorphisms) of the functor $E^{I^{(\alpha)}}_{\Ccal,\kappa}$, which is known to be an equivalence by \cref{prop:Result_for_Ialpha}, hence is itself an equivalence of category.

\end{proof}

\bibliography{/home/simon-henry/Dropbox/Latex/Biblio}{}

\begin{thebibliography}{10}

\bibitem{adamek1994locally}
Ji{\v{r}}{\'\i} Ad{\'a}mek and Ji{\v{r}}{\'\i} Rosick{\'y}.
\newblock {\em Locally presentable and accessible categories}, volume 189.
\newblock Cambridge University Press, 1994.

\bibitem{gregorybird}
Gregory~J Bird.
\newblock {\em Limits in 2-categories of locally presentable categories}.
\newblock PhD thesis, PhD thesis, University of Sydney. Circulated by the
  Sydney Category theory seminar, 1984.
\newblock \url{http://maths.mq.edu.au/~street/BirdPhD.pdf}.

\bibitem{carboni2001syntactic}
A~Carboni, MC~Pedicchio, and Ji{\v{r}}{\'\i} Rosick{\`y}.
\newblock Syntactic characterizations of various classes of locally presentable
  categories.
\newblock {\em Journal of Pure and Applied Algebra}, 161(1-2):65--90, 2001.

\bibitem{SGA4I}
{Deligne, P. and Boutot, JF and Grothendieck, A. and Illusie, L. and Verdier,
  JL}.
\newblock {\em {S{\'e}minaire de g{\'e}om{\'e}trie alg{\'e}brique du
  Bois-Marie, SGA 4 [1 over 2]: Theorie des topos et cohomologie etale des
  schemas}}.
\newblock Springer, 1973.

\bibitem{di2020gabriel}
Ivan Di~Liberti and Julia Ramos~Gonz{\'a}lez.
\newblock Gabriel--{U}lmer duality for topoi and its relation with site
  presentations.
\newblock {\em Applied Categorical Structures}, 28(6):935--962, 2020.

\bibitem{442055}
Simon Henry.
\newblock $\operatorname{Ind}({C^I})=\operatorname{Ind}{(C)^I}$?
\newblock MathOverflow.
\newblock \url{https://mathoverflow.net/q/442055} (version: 2023-03-03).

\bibitem{jacqmin2021stability}
Pierre-Alain Jacqmin and Zurab Janelidze.
\newblock On stability of exactness properties under the pro-completion.
\newblock {\em Advances in Mathematics}, 377:107484, 2021.

\bibitem{loregian2015coend}
Fosco Loregian.
\newblock Coend calculus.
\newblock {\em arXiv preprint arXiv:1501.02503}, 2015.

\bibitem{lurie2009higher}
Jacob Lurie.
\newblock {\em Higher topos theory}.
\newblock Number 170. Princeton University Press, 2009.

\bibitem{makkai1988strong}
Michael Makkai.
\newblock Strong conceptual completeness for first-order logic.
\newblock {\em Annals of pure and applied logic}, 40(2):167--215, 1988.

\bibitem{meyer1983completion}
Carol~Vincent Meyer.
\newblock {\em Completion of categories under certain limits}.
\newblock PhD thesis, McGill University, 1983.
\newblock
  \url{https://library-archives.canada.ca/eng/services/services-libraries/theses/Pages/item.aspx?idNumber=892984918}.

\bibitem{positselski2023notes}
Leonid Positselski.
\newblock Notes on limits of accessible categories.
\newblock {\em arXiv preprint arXiv:2310.16773}, 2023.

\bibitem{rogers2021toposes}
Morgan Rogers.
\newblock Toposes of monoid actions.
\newblock {\em arXiv preprint arXiv:2112.10198}, 2021.

\end{thebibliography}
\bibliographystyle{plain}

\end{document}